\documentclass{ifacconf}

\usepackage{graphicx}      
\usepackage{natbib}     
\usepackage{amsmath,amssymb,amsfonts}
\usepackage{algorithmic}
\usepackage{graphicx}
\usepackage{textcomp}
\usepackage{chemarrow}
\usepackage{subfig}
\newtheorem{definition}{Definition}
\newtheorem{proposition}{Proposition}
\newtheorem{proof}{Proof}

\begin{document}
\begin{frontmatter}

\title{Accurate control to run and stop chemical reactions via relaxation oscillators} 

\author[First]{Xiaopeng Shi} 
\author[First]{Chuanhou Gao} 
\author[second]{Denis Dochain}

\address[First]{School of Mathematical Sciences, Zhejiang University, Hangzhou, 310058,China (e-mail: \{12035033; gaochou\}@zju.edu.cn)}
\address[second]{ICTEAM, UCLouvain, B ˆatiment Euler, avenue
Georges Lemaˆıtre 4-6, 1348 Louvain-la-Neuve, Belgium (e-mail: de-
nis.dochain@uclouvain.be)}

\begin{abstract}                
Regulation of multiple reaction modules is quite common in molecular computation and deep learning networks construction through chemical reactions, as is always a headache for that sequential execution of modules goes against the intrinsically parallel nature of chemical reactions.
Precisely switching multiple reaction modules both on and off acts as the core role in programming chemical reaction systems. Unlike setting up physical compartments or adding human intervention signals, we adopt the idea of chemical oscillators based on relaxation oscillation, and assign corresponding clock signal components into the modules that need to be regulated. 
This paper mainly demonstrates the design process of oscillators under the regulation task of three modules, and provides a suitable approach for automatic termination of the modules cycle. We provide the simulation results at the level of ordinary differential equation and ensure that equations can be translated into corresponding chemical reaction networks. 
\end{abstract}


\end{frontmatter}

\section{Introduction}
The engineering of biological systems is usually based on modular and hierarchical design, wherein a system is described as composition of simpler subsystems whose properties are fully understood (\cite{qian2018programming}). The resulting molecular computation can not avoid a key problem, that is, how can subsystems (or modules) be automatically sequenced and continue to perform their respective functions after being splintered into a whole system. A naive idea is to set up physical compartments to separate the modules (\cite{blount2017feedforward}), but this does not always work in a biochemical context, especially when applied to cellular programming and synthetic gene circles (\cite{qian2018programming}). Another design that is more suitable for biochemical environment is to build chemical oscillators in order to generate periodic intervention clock signals (\cite{jiang2011,vasic2020,9707599,shi2022design}). These intervention signals serve in the biochemical system as catalysts and regulate the switch of the module according to the periodic fluctuations of their concentrations.\\
\\
Just like stability and persistence in chemical and biological environment (\cite{zhang2022stability,zhang2020persistence}), oscillation, as a critical dynamical property, has also been studied for a long time (\cite{tyson2013,forger2017biological,gonze2021goodwin}). When it comes to design chemical oscillators to switch modules, we have a harsh choice of how the oscillation behaves. The chemical oscillator models used in many previous works either lack a clear understanding of the oscillation mechanism (\cite{jiang2011}), or are too radical in the choice of parameters and initial values (\cite{9707599}). In the practical requirements of programming chemical reactions, we often need to have a clear grasp of the period of the clock signals, and this determines how to allocate a reasonable execution time for a module according to the speed at which the reactions in the module reach equilibrium. Choice of parameters and initial concentration of components also need to have certain manoeuverability and robustness in practice. Another problem is that most work introducing chemical oscillators only consider how to make the oscillator guide the cycle of modules to occur, while few pay attention on termination of cycle. \\
\\
In our module design, each chemical reaction module corresponds to a calculation instruction similar to the one in computer science (\cite{vasic2020}), and the clock signal demarcates the execution order of these reaction instructions, so cycle termination is an integral part of the instruction. However, it is natural that an oscillation cannot terminate itself, otherwise it would not be an oscillation. We need to prepare a strategy for the system to terminate oscillations on its own while introducing chemical oscillators for module regulation.
In our previous work, we proposed a chemical oscillator structure based on relaxation oscillation, and obtained satisfactory results in the oscillation mechanism and parameter selection (\cite{shi2022design}). However, that work only focused on the adjustment task of two modules and failed to solve the problem of automatic cycle termination well. In this article we explain that the number of cycle executions should also be calibrated separately as a reaction module, so the regulation task faced by chemical oscillators is intrinsically multi-module, which
is also common in biological environment.\\
\\
Our discussion expands from ordinary differential equations(ODEs) for which we select the appropriate oscillator based on the oscillatory structure of the ODEs system. We require that the state points of ODEs that correspond to concentration of chemical components, are in the first quadrant, and then values of these elements corresponding to intervention signals can sufficiently go close to zero on the low amplitude segment. Based on these considerations, we choose relaxation oscillation as the basic architecture of our model. On the one hand, it is easy to make corresponding design based on the clear oscillation mechanism, and on the other hand, it helps to guarantee the robustness of the oscillators (\cite{krupa2001relaxation}). Besides, we try to ensure that the established ODEs are in specific polynomial form so that they can be translated back into chemical reaction networks (CRNs) through \textit{mass-action kinetics}. Research on properties of chemical reaction networks and their computational power has also been fruitful (\cite{wu2020lyapunov,fages2017strong,chalk2019composable,vasic2020}). DNA strand displacement cascades can be utilized to finally implement CRNs into real chemistry (\cite{soloveichik2010dna}), so our work focuses on the design of oscillators at the dynamical system level.\\
\\
This paper is organized as follows. Basic concepts about CRN and our relaxation oscillator model is given in \uppercase\expandafter{\romannumeral2}. Section \uppercase\expandafter{\romannumeral3} exhibits the main results of constructing relaxation oscillators for three modules and approach for cycle termination. Finally, section \uppercase\expandafter{\romannumeral4} is dedicated to conclusion and discussion of the whole paper.

\section{Basic concept}
In this section, some basic concepts of CRNs and relaxation oscillator model we use are provided.

\subsection{Chemical Reaction Networks}
Chemical reaction networks often consist of $n$ species $X_{1},...,X_{n}$ and $m$ reactions $R_{1},...,R_{m}$ in which a reaction is just like
\begin{equation*}
R_{i}: a_{i1}X_{1}+\cdots +a_{in}X_{n} \overset{k_{i}}{\rightarrow} b_{i1}X_{1}+\cdots +b_{in}X_{n}\ ,
\end{equation*}
for $k_{i}$ is the rate constant of this reaction.\\
It is customary to refer to species by a capital letter and to denote the concentration of that species by the corresponding lowercase letter. For example, $x_{1}=x_{1}(t)$ refers to concentration of species $X_{1}$ along time $t$. Then the kinetic model of species set ${X_{1},...,X_{n}}$ can be expressed as
\begin{equation*}
    \dot{\textbf{x}} = \Gamma \cdot v\left(\textbf{x}\right).
\end{equation*}
where $\Gamma$ usually refers to coefficient matrix and satisfies $\Gamma_{ij} = b_{ij}-a_{ij}$, while rate function $v\left(\textbf{x}\right)$ depends on kinetic assumption. With \textit{mass-action kinetics}, the rate function is $v\left (\textbf{x} \right ) = \left ( k_{1}\prod_{i=1}^{n}x_{i}^{a_{i1}},...,k_{m}\prod_{i=1}^{n}x_{i}^{a_{im}}  \right )^{\top}$.

It is important to underline that the kinetic assumption we choose restrict the expression of ODEs to polynomials with specific rules. Therefore, special attention should be paid when designing the oscillator model. Obviously, there is no reaction network corresponds to equation like $\dot{x}= x-y$ for that species $X$ cannot be consumed if it fails to appear in the left of reaction (\cite{hangos2011mass}). we modify the possible similar structures when designing the relaxation oscillator model, as we will exhibit in the next subsection.

\subsection{Relaxation Oscillator Model}
Relaxation Oscillations are a type of periodic solutions found in singular systems and ubiquitous in systems modelling chemical and biological phenomena (\cite{mishchenko2013differential,murray2002mathematical,epstein1998introduction}), related research on complex dynamic behavior, including canard explosion (\cite{krupa2001relaxation}) and synchronization (\cite{fernandez2020symmetric}), is also flourished. Relaxation oscillation is robust to the initial point selection, while we still have to adapt the common relaxation oscillation structure to suit our design needs. Another worth noticing is that relaxation oscillators can not directly satisfy the requirements for intervention signals, as it is difficult to get the oscillator close enough to zero in the low amplitude segment. Based on these considerations, we adjust the structure of the 2-dimension relaxation oscillation model and couple it with a symmetric truncated subtraction operation system. Model presented in this paper is similar to the one we constructed in (\cite{shi2022design}) as follows:

 \begin{equation}
    \begin{aligned}
    \frac{\mathrm{d} x}{\mathrm{d} t} &= \eta_{1}(-x^3+6x^2-9x+5-y)x/\epsilon\ , \\
    \frac{\mathrm{d} y}{\mathrm{d} t} &= \eta_{1}(x-\rho)y\ , \\
    \frac{\mathrm{d} u}{\mathrm{d} t} &= \eta_{2}(p-u-cuv)\ , \\
    \frac{\mathrm{d} v}{\mathrm{d} t} &= \eta_{2}(x-v-cuv)\ .
    \end{aligned}
\end{equation}

Subsystem $\sum_{xy}$ is a specific relaxation oscillation model rewritten from (\cite{krupa2001relaxation}) in which We uniformly multiply the corresponding variable on the right side of ODEs to satisfy the requirement of transforming the equation back into CRN according to \textit{mass-action kinetics}. The equilibrium points thus introduced do not affect the expression of the phase plane portrait in the strictly first quadrant which is shown in Fig.1. Parameter $\epsilon$ is small enough to ensure the existence of relaxation oscillation whose trajectory convergences to the one consisting of both solid and dashed red lines in Fig.1. Parameter $\eta_{1}$ and $\rho$ respectively control overall length of oscillatory period of $x$ and the ratio between high and low amplitude segments. \\
 \begin{figure}[ht]
 	\centerline{\includegraphics[width=\columnwidth]{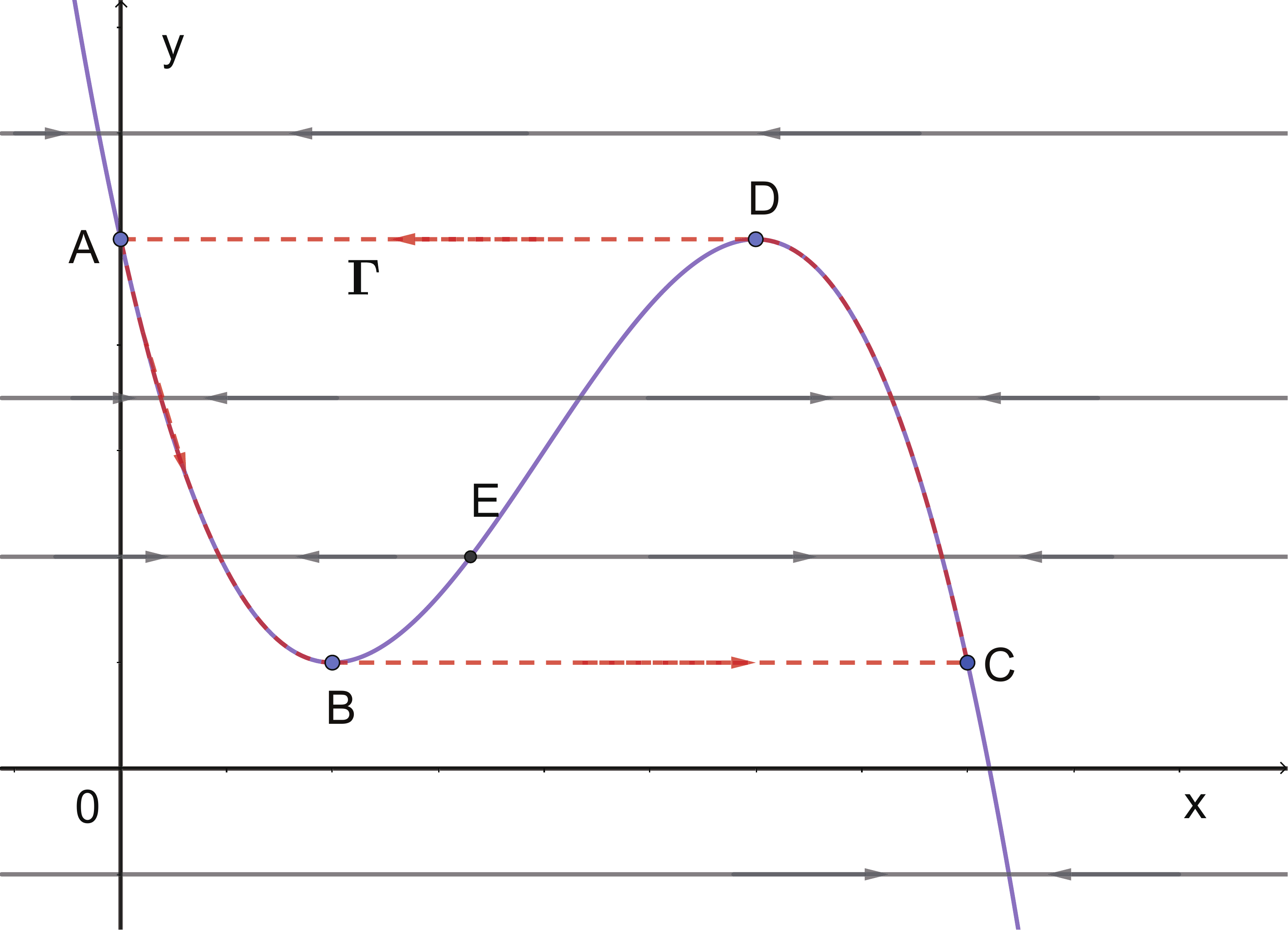}}
 	\caption{Phase plane portrait of subsystem $\sum_{xy}$.}
 	\label{fig1}
 \end{figure}
 \\
Subsystem $\sum_{uv}$ is a truncated subtraction module adjusted from the one in (\cite{vasic2020}). We use this subsystem to set the value of $x$ close to zero at low amplitudes via a large enough parameter $c$. The structure of the subsystem $\sum_{uv}$ shows a certain degree of symmetry, so as to ensure that $u$ and $v$ can act as a pair of symmetric clock signals while they are affected by the oscillatory element $x$ and oscillate along with it(See (\cite{shi2022design})). Parameters $\eta_{2}$ and $p$ determine the response and amplitudes of $u$ and $v$.\\
\\
What needs to be emphasized is that we only restrict the value of the parameters to a certain extent, and each parameter is chosen to be explained accordingly. However, selection of initial points is free for that as long as point E is avoided (see Fig.1), the initial points in the first quadrant will cause the trajectory to converge to the periodic orbit of relaxation oscillation (\cite{shi2022design}). This ensures the feasibility of our oscillator model in practical applications. With $\eta_{1}=0.1, \epsilon=0.001, \rho=2.1, \eta_{2}=10, p=2, c=5000$ and initial point $(1,1,0,0)$, we get the simulation diagram in Fig.2.\\
 \begin{figure}[ht]
 	\centerline{\includegraphics[width=\columnwidth]{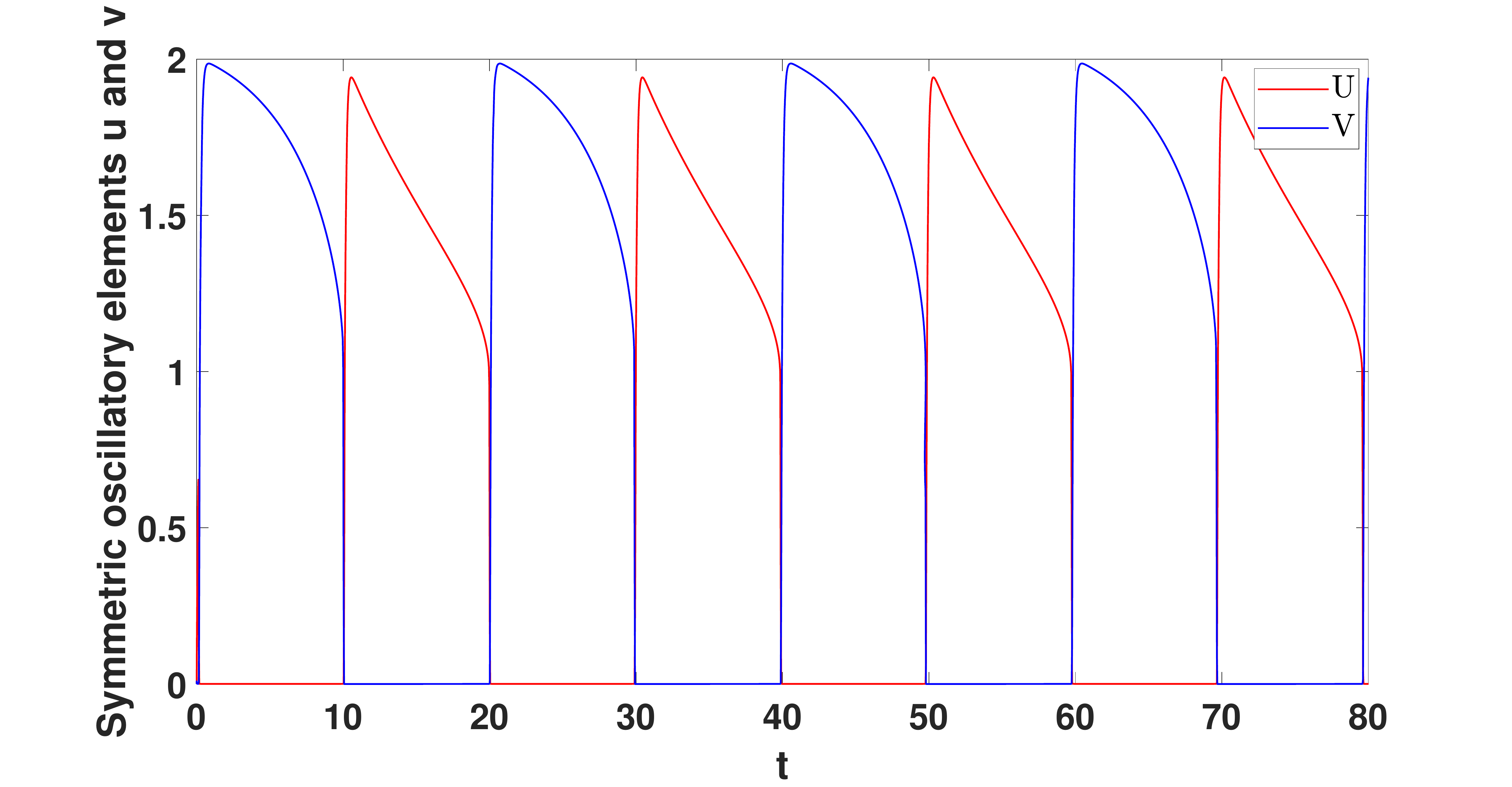}}
 	\caption{Simulation diagram of $u$ and $v$.}
 	\label{fig2}
 \end{figure} 
\\
Corresponding CRNs are shown as follows:
\begin{align*}
    4X&\overset{\eta_{1}/\epsilon }{\rightarrow}3X\ ,\\
3X&\overset{6\eta_{1}/\epsilon }{\rightarrow}4X\ , \\
2X&\overset{9\eta_{1}/\epsilon }{\rightarrow}X\ , \\
X&\overset{5\eta_{1}/\epsilon }{\rightarrow}2X\ , \\
X+Y&\overset{\eta_{1}/\epsilon }{\rightarrow}Y\ , \\
X+Y&\overset{\eta_{1}}{\rightarrow}2Y\ , \\
Y&\overset{\eta_{1}\rho}{\rightarrow}\varnothing\ , \\
\varnothing&\overset{\eta_{2}p}{\rightarrow}U\ , \\
U&\overset{\eta_{2}}{\rightarrow}\varnothing\ , \\
U+V&\overset{\eta_{2}c}{\rightarrow}V\ , 
\end{align*}
\begin{align*}
X&\overset{\eta_{2}}{\rightarrow}V+X\ , \\
V&\overset{\eta_{2}}{\rightarrow}\varnothing\ , \\
U+V&\overset{\eta_{2}c}{\rightarrow}U\ .
\end{align*}
and in the end of this subsection, we repeat the definition of (symmetric) clock signals just like (\cite{shi2022design}).

\begin{definition}
	A species $U$ is called clock signal for that its concentration $u(t)$ oscillates over time, which reaches zero or close enough during some part in a oscillation period and immediately goes beyond zero during the rest. 
\end{definition}
\begin{definition}
	A pair of clock signals $U$ and $V$ are called symmetric for that when $U$ oscillates to zero or close enough, $V$ goes strictly beyond zero, and vice versa.
\end{definition}

\section{Main Results}
In this section, we try to take advantage of system (1) and build oscillators for multiple reaction modules regulation. We view the designed oscillators as a mean of controlling the chemical system so that the chemical reactions, which are otherwise scattered and simultaneous, proceed in the desired order and terminate at any time we want. In principle, our design idea is applicable to any number of reaction modules regulation, but for simplicity of discussion, we only demonstrate the construction process of oscillators in the context of three modules regulation.

\subsection{Relaxation Oscillators for Three Modules}
In (\cite{shi2022design}), we tackled with the loop of two modules and used a pair of symmetric clock signals $U$ and $V$ according to elements $u$ and $v$ in system (1). When it turns to three modules and even more, an intuitive idea is to decompose the regulation task and construct several pairs of clock signals. Assume that we have three modules to be cycled as $Module_{1}$, $Module_{2}$ and $Module_{3}$, and in each loop these three modules perform the corresponding operations in turn (We shall provide a specific example in the next subsection). We can easily image that two pairs of clock signals have separately been constructed as system (1), then we utilize species $V_{1}$ as catalyst for overall reactions in $Module_{1}$ while $U_{1}$ for overall reactions in both $Module_{2}$ and $Module_{3}$. So in the first half of each period, concentration of $V_{1}$ (so as $v_{1}$ in ODEs) is strictly positive which starts $Module_{1}$, while concentration of $U_{1}$ goes close to zero, shutting down $Module_{2}$ and $Module_{3}$. Situation is reversed in the second half of the period. Based on this, we just call for the other pair of clock signals $U_{2}$ and $V_{2}$ to control the sequence of $Module_{2}$ and $Module_{3}$ when $Module_{1}$ is shutting down. We give a schematic diagram for this in Fig.3.\\
\begin{figure}[ht]
\subfloat[Schematic diagram of the clock signals inserted into the corresponding modules.]{\includegraphics[width=0.47\columnwidth]{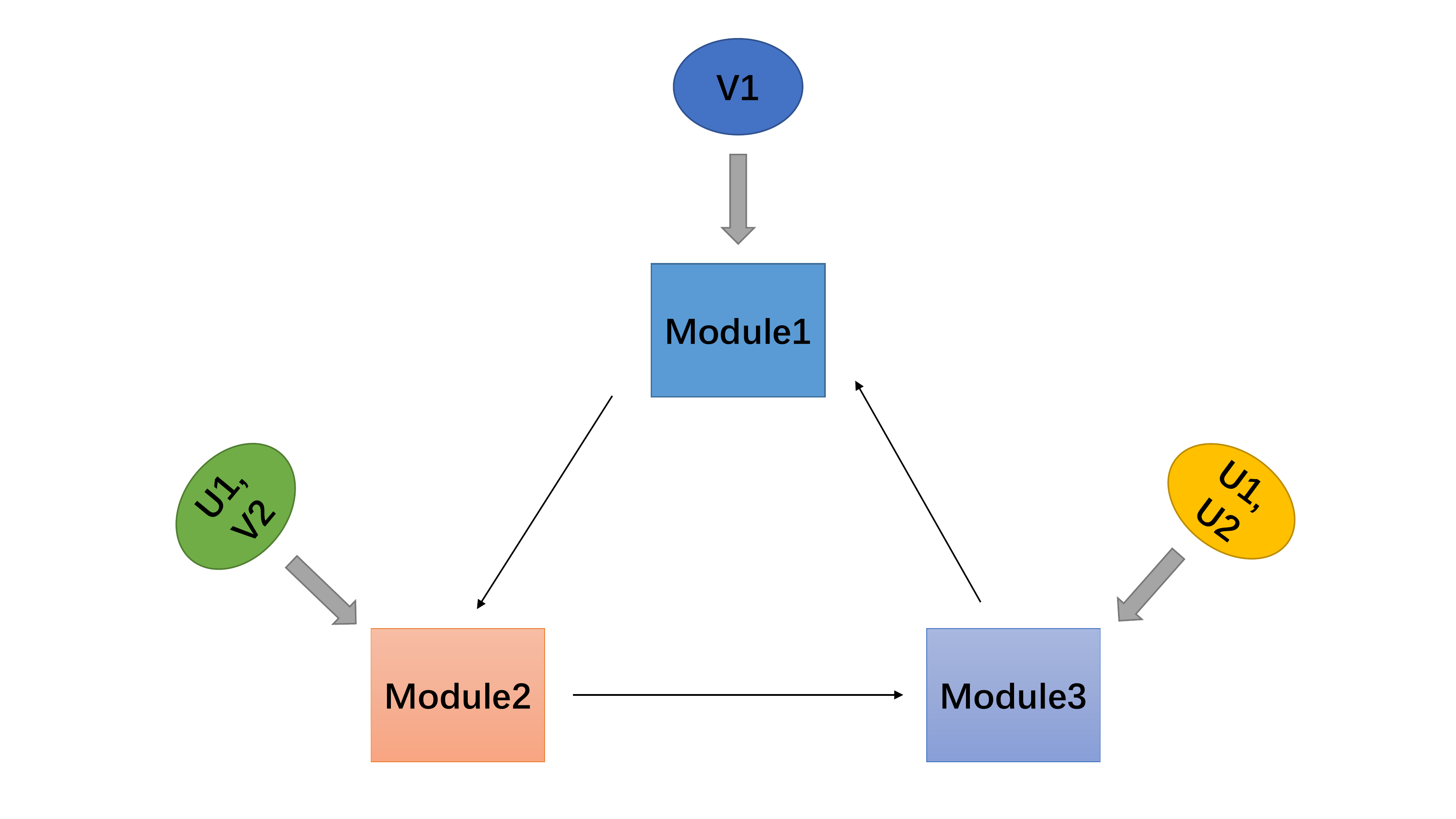}}
\hfill
\subfloat[Schematic diagram of module execution order.]{\includegraphics[width=0.47\columnwidth]{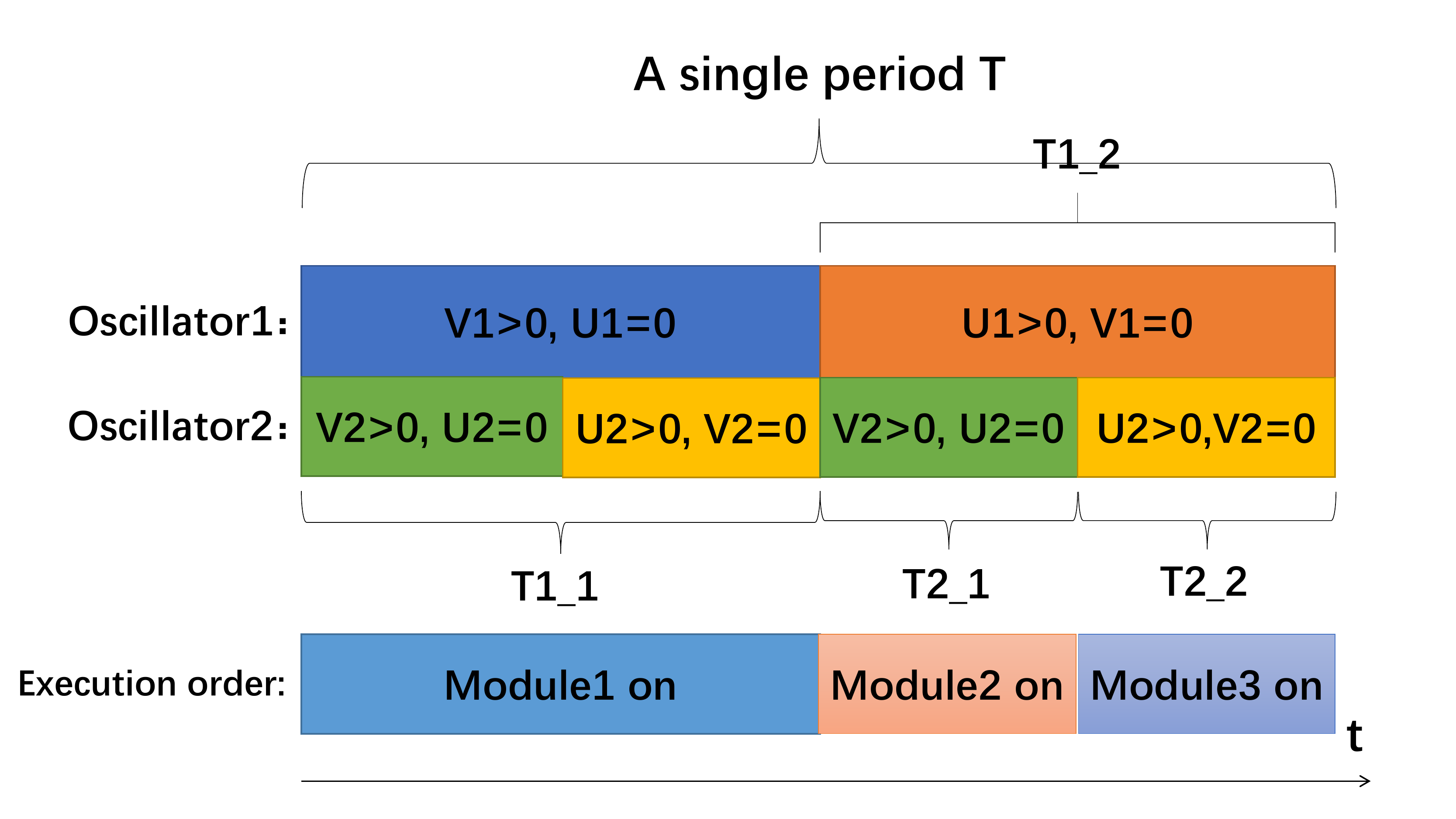}}
\caption{Schematic diagram for three-module regulation in a single period $T$.}
\end{figure}

Obviously, an important prerequisite for this idea to work is that the period of $u$ and $v$ in system (1) can be regulated. The following Proposition 1 is given to estimate this period.
\begin{proposition}\label{proposition1}
  With $\epsilon$ small enough and a large $\eta_{2}$, the period of $u$ and $v$ in system (1) is estimated as $T_{1}+T_{2}$ as follows:
\begin{align*}
    T_{1}\approx&\int_{4}^{3}\frac{-3x^{2}+12x-9}{\eta_{1}(x-\rho)(-x^{3}+6x^{2}-9x+5)}dx \ ,\\
    T_{2}\approx&\int_{0}^{1}\frac{-3x^{2}+12x-9}{\eta_{1}(x-\rho)(-x^{3}+6x^{2}-9x+5)}dx \ .
\end{align*}
\end{proposition}

\begin{proof}
 Theorem 4.2 in (\cite{shi2022design}) provides a strict form of period of $x$ as
  \begin{align*}
        T_{1}= \int_{x_{A}}^{x_{m}}\frac{(f'(x)-\frac{\mathrm{d} \psi _{1}}{\mathrm{d} x}(x,\epsilon ))dx}{\eta_{1}(x+\mu(f(x)-\psi _{1}(x,\epsilon ))+\lambda)(f(x)-\psi _{1}(x,\epsilon ))}\ ,
    \end{align*}
    \begin{align*}
       T_{2}= \int_{x_{C}}^{x_{M}}\frac{(f'(x)+\frac{\mathrm{d} \psi _{2}}{\mathrm{d} x}(x,\epsilon ))dx}{\eta_{1}(x+\mu(f(x)+\psi _{2}(x,\epsilon ))+\lambda)(f(x)+\psi _{2}(x,\epsilon ))}\ ,
    \end{align*}
which corresponds to subsystem $\sum_{xy}$ as
    \begin{align*}
    \frac{\mathrm{d} x}{\mathrm{d} t} &= \eta_{1}(f(x)-y)x/\epsilon\ , \\
    \frac{\mathrm{d} y}{\mathrm{d} t} &= \eta_{1}(x-\rho)y\ ,
    \end{align*}
and put some constraints on the expression of $f(x)$. While in this paper we uniformly let $f(x)=-x^3+6x^2-9x+5$, then the lower and upper bounds of the integral also correspond to specific values. Terms about $\psi_{i}(i=1,2)$ are closely related to the periodic orbit of the relaxation oscillation, which are controlled by parameter $\epsilon$. As a rough estimate of the period, we make sure that the error caused by erasing these terms is acceptable by picking a small enough $\epsilon$. So the expression of $T_{1}$ and $T_{2}$ can be directly used as an approximate estimate of the period of $x$ in system (1). \\
\\
It is quite difficult to estimate the period of $u$ and $v$ directly in their subsystem, so we set two parameters to control the effect of coupling oscillation: Large enough $c$ makes every instantaneous equilibrium of $u$ (or $v$) in the subsystem approximate the maximum value between $p-x$ (or $x-p$) and 0. Large enough $\eta_{2}$ ensure that $u$ and $v$ response quickly towards the oscillation of $x$, showing almost synchronous changes.
In fact, subsystem generating $u$ and $v$ is a symmetrical modification to the subtraction module used in (\cite{vasic2020}), we can therefore expect our reaction module to converge exponentially quickly as they claimed, though neither we nor they found the analytical solution. However, under appropriate parameter declarations, we directly utilize $T_{1}$ and $T_{2}$ exhibited in Proposition 1 to estimate period of our oscillators $u$ and $v$. 
 $\hfill\blacksquare$
\end{proof}

What is clear by Proposition \ref{proposition1} is that under the structure of system (1), parameter $\rho$ influences the ratio $T_{1}/T_{2}$ and $T_{i}(i=1,2)$ is inversely proportional to $\eta_{1}$. Now let us come back to our original aim. We construct another system (1) with a doubled $\eta_{2}$, clock signals $U_{2}$ and $V_{2}$ emerge and whose period is exactly half the one of $U_{1}$ and $V_{1}$. 
The complete system designed to generate whole clock signals for three-module regulation is concluded as follows:
\begin{equation}
\begin{aligned}
    \frac{\mathrm{d} x_{1}}{\mathrm{d} t} &= \eta_{1}(-x_{1}^3+6x_{1}^2-9x_{1}+5-y_{1})x_{1}/\epsilon\ , \\
    \frac{\mathrm{d} y_{1}}{\mathrm{d} t} &= \eta_{1}(x_{1}-\rho)y_{1}\ , \\
    \frac{\mathrm{d} u_{1}}{\mathrm{d} t} &= \eta_{2}(p-u_{1}-cu_{1}v_{1})\ , \\
    \frac{\mathrm{d} v_{1}}{\mathrm{d} t} &= \eta_{2}(x_{1}-v_{1}-cu_{1}v_{1})\ , 
\end{aligned}
\end{equation}
\begin{equation}
\begin{aligned}
    \frac{\mathrm{d} x_{2}}{\mathrm{d} t} &= 2\eta_{1}(-x_{2}^3+6x_{2}^2-9x_{2}+5-y_{2})x_{2}/\epsilon\ , \\
    \frac{\mathrm{d} y_{2}}{\mathrm{d} t} &= 2\eta_{1}(x_{2}-\rho)y_{2}\ , \\
    \frac{\mathrm{d} u_{2}}{\mathrm{d} t} &= \eta_{2}(p-u_{2}-cu_{2}v_{2})\ , \\
    \frac{\mathrm{d} v_{2}}{\mathrm{d} t} &= \eta_{2}(x_{2}-v_{2}-cu_{2}v_{2})\ .
\end{aligned}
\end{equation}

System above generates two pairs of clock signals $U_{1}$, $V_{1}$ and $U_{2}$, $V_{2}$. To regulate the sequence of three modules, we put species $v_{1}$ into overall reactions in $Module_{1}$ as catalyst, both $U_{1}$ and $U_{2}$ into overall reactions in $Module_{2}$ and last, both $U_{1}$ and $V_{2}$ into overall reactions in $Module_{3}$ as catalyst. We present an example of module regulation, as well as accurate cycle termination in molecular computation in the next subsection. In the numerical simulation, we choose $\epsilon=0.001$ and $\eta_{2}=10$, with other parameters same in previous section, oscillation periods can be easily compared in Fig.4.

 \begin{figure}[ht]
 	\centerline{\includegraphics[width=\columnwidth]{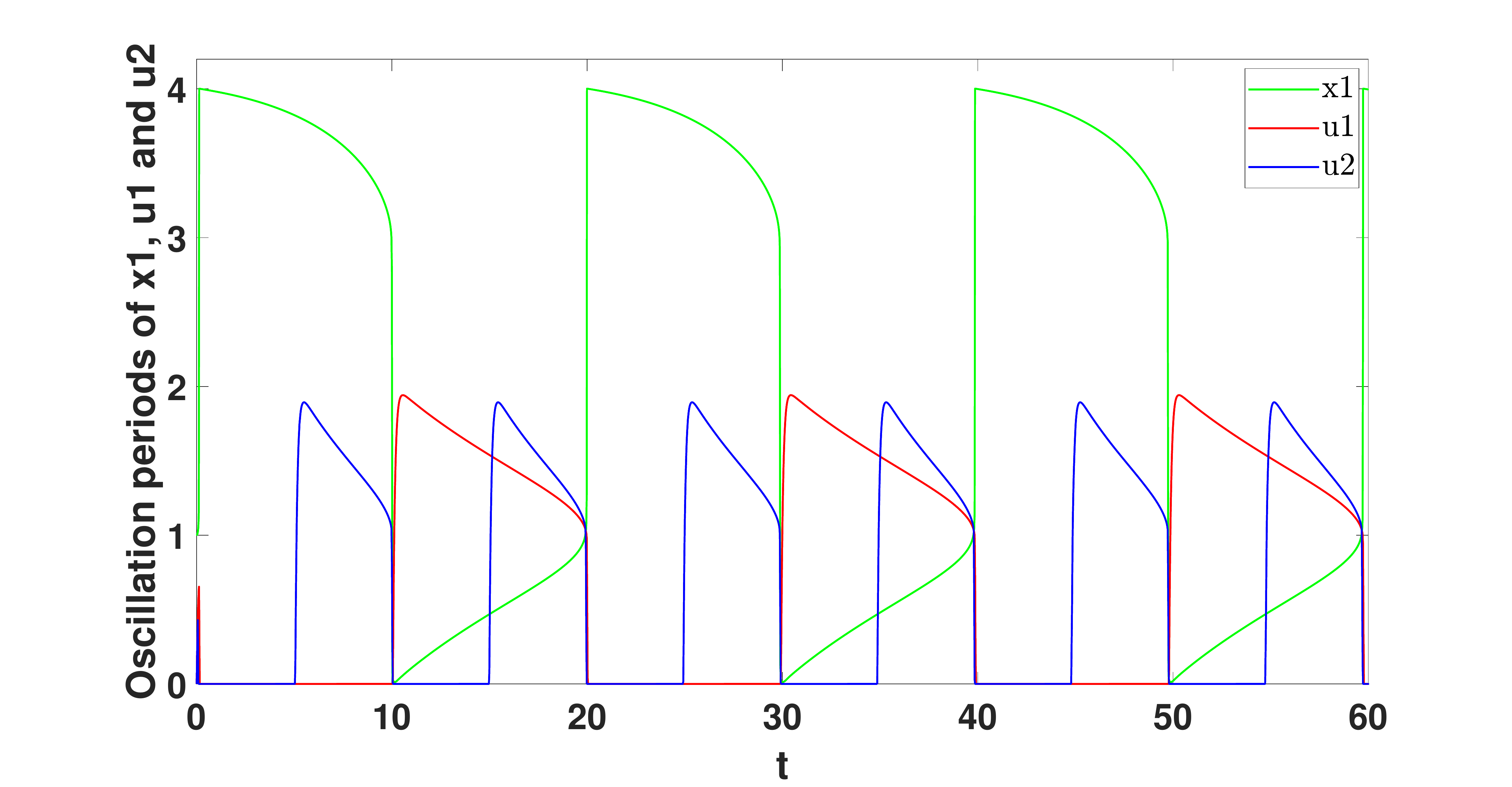}}
 	\caption{Compare between periods of $x$,$u$ and $v$.}
 	\label{fig4}
 \end{figure}

\subsection{Cycle Termination Based on Multi-module Regulation}
Cycle termination of chemical reaction modules means shutting down all reactions under appropriate circumstances. In the related work of molecular computation and designing deep learning networks through chemical reactions (\cite{9707599,vasic2020}), we usually intercept the concentration of some species in the system at a specific moment as the system output. So in addition to inserting in intervention signals to guide the reactions in order, we also need to provide a solution for shutting down all the reactions, and to some extent, forcing the reaction system to stop. Our cycle termination strategy is to track the number of cycles performed with a particular species, and when the number reaches a predetermined value, concentration of this species returns to zero, thereby shutting down all the reactions catalyzed by it.\\

In our previous work, we tried to make cycle of the two modules stop automatically after a preset number $n$ of executions, so we first constructed a truncated subtraction module with exponential velocity to complete convergence, this module was then used to produce a component that measures the instantaneous difference between the counting unit $x_{1}$ in our counter model and $n$, and finally we added this component into overall modules. After the given number $n$ of loops, the component would reach zero as we supposed, turning the whole system down. While, although the truncated subtraction module we designed maintained an exponential rate of convergence, it still caused that the component measuring the difference could not respond to the concentration increase of counting unit $x_{1}$ with no error, modules to be shut down would therefore continue to perform operations for a short time after the $n$th cycle. This paper provides a better approach to solve this problem.\\
\\
First, we recall our counter model and truncated subtraction module in previous work. 
We sort the following truncated subtraction module as $Module_{1}$. In this module, species $X$ is just the component to measure the difference between counter unit $Y$ and the given loop time $N$.
\begin{align*}
Module_{1}:
N+X&\rightarrow N+2X\ , \\
Y+X&\rightarrow Y\ , \\
2X&\rightarrow X\ .
\end{align*}
Our counter model consists of following two modules as $Module_{2}$ and $Module_{3}$, species $L$ acts as a constant, then result of the sequential execution of these two modules is the self-increment of $Y$'s value (often expressed as concentration in the context of chemistry). 
\begin{minipage}{0.5\columnwidth}
\begin{align*}
    Module_{2}:
Y&\rightarrow Y+Z\ , \\
L&\rightarrow L+Z\ , \\
Z&\rightarrow \varnothing \ .
\end{align*}
\end{minipage}
\begin{minipage}{0.5\columnwidth}
\begin{align*}
    Module_{3}:
Z&\rightarrow Y+Z\ , \\
Y&\rightarrow \varnothing \ .
\end{align*}
\end{minipage}
Then ODEs for this three-module model are concluded as follows:
 \begin{equation}
    \begin{aligned}
    \frac{\mathrm{d} x}{\mathrm{d} t}&= \eta_{3}(n-y-x)xv_{1}\ ,\\
    \frac{\mathrm{d} y}{\mathrm{d} t}&= \eta_{4}(z-y)u_{1}u_{2}x\ ,\\
    \frac{\mathrm{d} z}{\mathrm{d} t}&= \eta_{4}(y+1-z)u_{1}v_{2}x\ .
    \end{aligned}
\end{equation}
Elements $u_{1}$, $u_{2}$, $v_{1}$ and $v_{2}$ come from our oscillatory system, parameters $\eta_{3}$ and $\eta_{4}$ control convergence speed of corresponding ODE. Combine clock signals with modules to be regulated, with $\eta_{1}=0.1, \epsilon=0.001, \rho=2.1, \eta_{2}=10, p=2, c=5000, \eta_{3}=500, n=4, \eta_{4}=1$, and initial point $(x_{1}, y_{1}, u_{1}, v_{1}, x_{2}, y_{2}, u_{2}, v_{2}, x, y, z)=(1,1,0,0,1,1,0,0,4,0,0)$, we get the simulation diagram for counter unit $Y$ and auxiliary unit $X$ in Fig.5 (a). Our aim is to stop the counter unit $Y$ from self-increment after $n=4$ loops and compared with Fig.5 (b), which is our previous result in (\cite{shi2022design}), cycle termination in this paper has a more accurate result.

\begin{figure}[ht]
\subfloat[Result in this paper.]{\includegraphics[width=0.5\columnwidth]{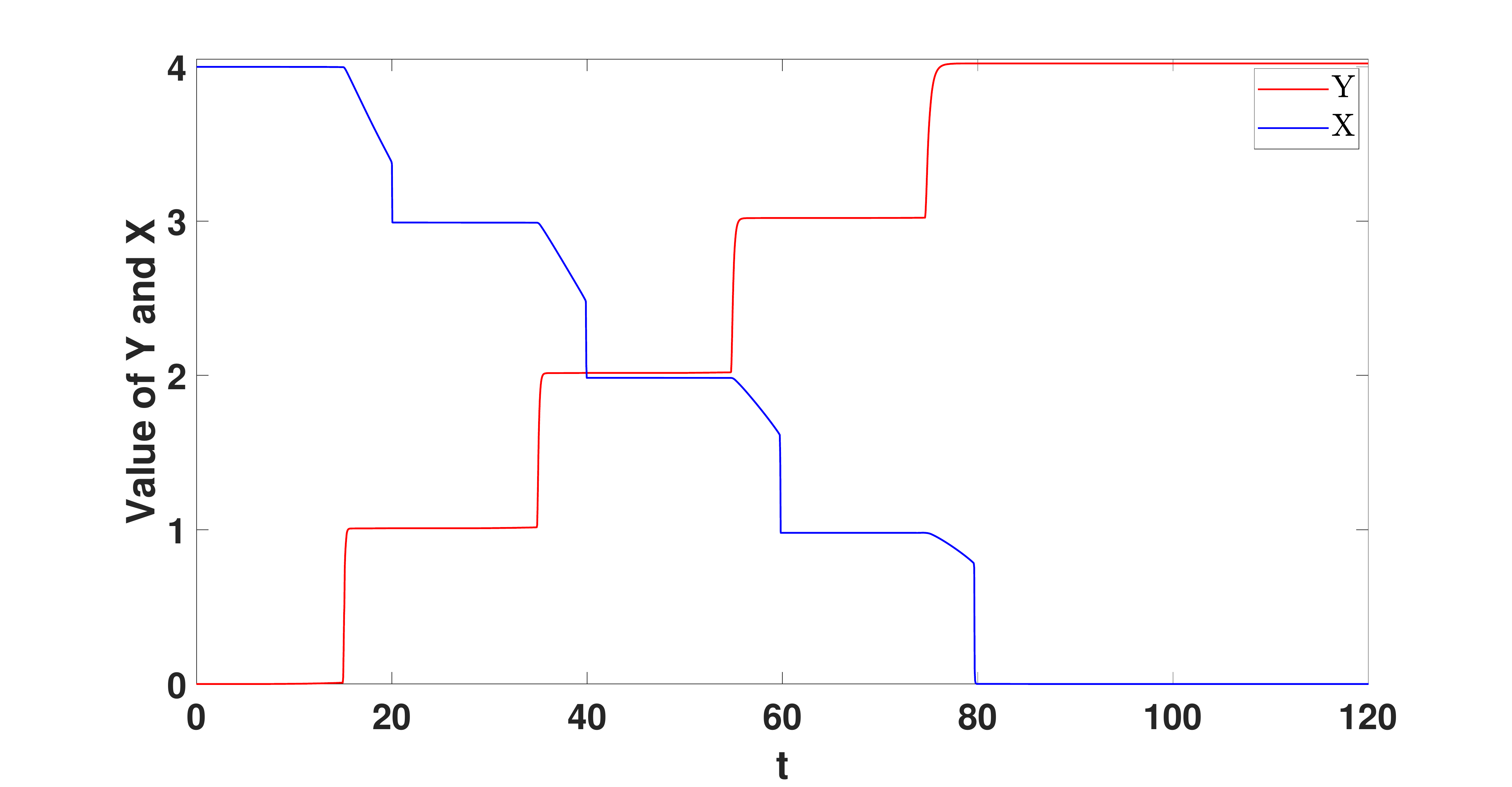}}
\hfill
\subfloat[Result of our previous work.]{\includegraphics[width=0.5\columnwidth]{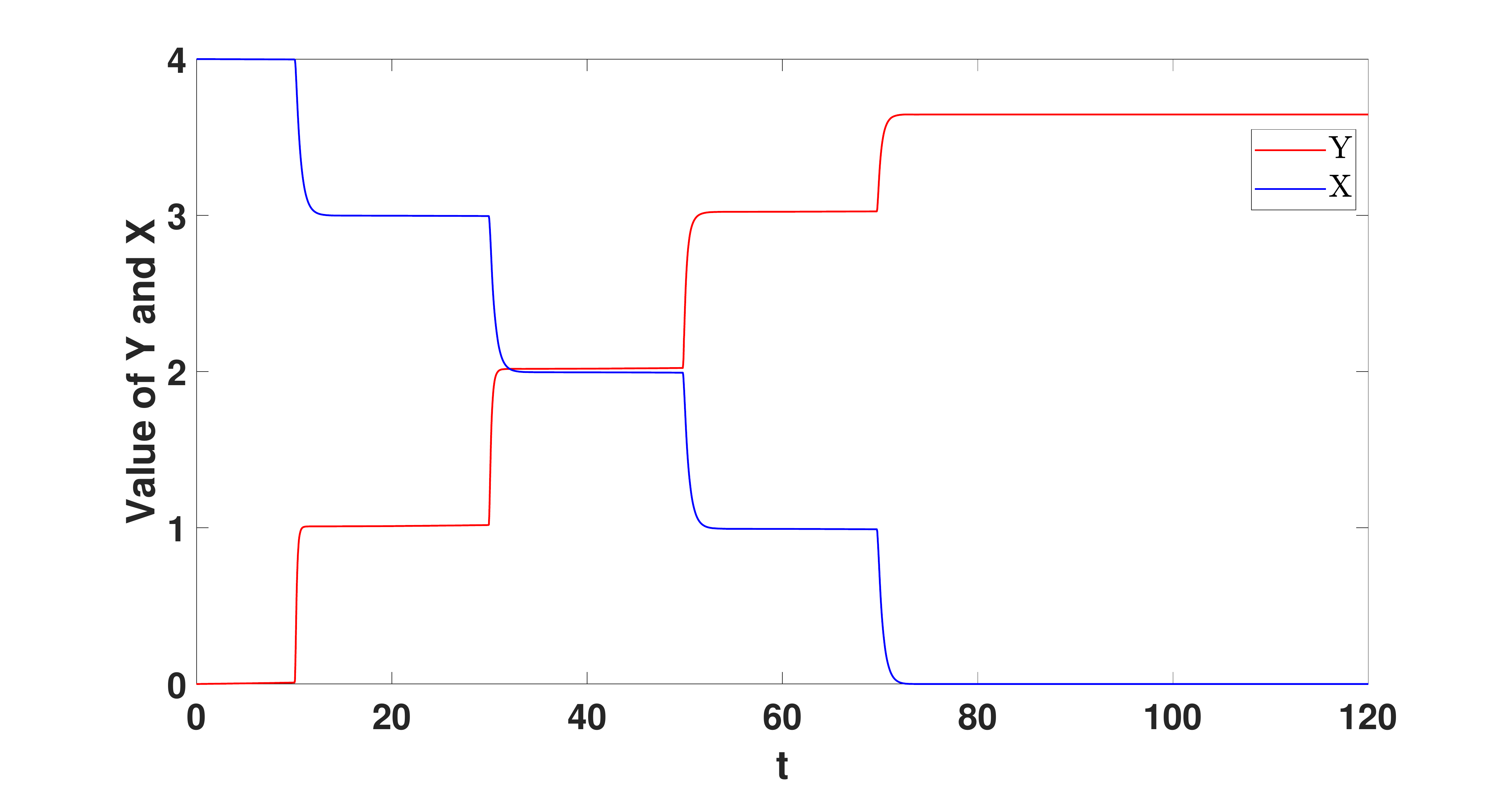}}
\caption{Simulation diagram for counter unit $Y$ and auxiliary unit $X$.}
\end{figure}

 In principle, our oscillator regulation mechanism is suitable for the regulation task of arbitrarily number of modules, and we outline the general points as following proposition:
\begin{proposition}
  For adjustment task of $m\ (m\in \mathbb{Z}_{+},\ m\geq 2)$ reaction modules:
  \begin{itemize}
      \item Our oscillator model uses system (1) as the underlying architecture to generate each pair of symmetric clock signals and system (4) as the cycle termination mechanism. Since we have to put in a separate module like the $Module_{1}$ to track the number of loops, we are actually dealing with regulation of $m+1$ modules sorted as $Module_{1}, Module_{2}, ..., Module_{m+1}$. So we need $m$ oscillators as system (2), system (3) and so on to fit together.
      \\
      \item Our oscillators decompose the multi-module regulation task and the $k$th oscillator is specifically used to regulate the execution order of $Module_{k}$ and the large module composed of $Module_{k+1}, ..., Module_{m+1}$ ($1\leq k\leq m$). And the period of these $m$ oscillators is halved successively, that is to say, the period of clock signals $U_k$ and $V_k$ should decay to $\frac{1}{2^{k-1}}$ of the one of $U_1$ and $V_1$.
      \\
      \item We continue to use species $Y$, $Z$ and $X$ mentioned in the previous subsection as the components for counting and terminating cycle, and let the clock signal $U_1$, $V_1$, $U_2$ and $V_2$ guide the normal operation of this counter model as system (4) in parallel while regulating the execution order of the target module. Thus, the cycle time of the complete $m+1$ modules is determined by the period of $U_1$ and $V_1$.
      \\
      \item Our clock signals participate in module regulation with the form of catalyst, so they do not affect the exponential convergence of the original reaction modules. 
  \end{itemize}
\end{proposition}
We omit the further explanation of above proposition and the complete expression of CRNs, and leave the related tasks of the actual chemical experiments to the DNA strand displacement cascades (\cite{soloveichik2010dna}).

\section{Conclusion}
Design of our chemical oscillators is modular to a certain extent, subsystem exhibiting relaxation oscillation assembles with truncated subtraction modules, producing clock signals we want. Our oscillators can be viewed as a control for synthetic biology with the idea of engineering modularity (\cite{qian2018programming}), or exactly run and stop chemical reactions. And different from the model mentioned in (\cite{vasic2020}) and (\cite{9707599}), we pay more attention on mechanism for generating oscillation. Our model is able to change the amplitude and period of oscillators according to the specific regulatory needs and is more robust towards the selection of parameters and initial points. \\
\\
This paper demonstrates an example of construct chemical oscillators for three modules by superimposing two relaxation oscillators whose periods are multiples of each other. What calls for special attention is that the periods of oscillators corresponding to the three modules are not exactly same. Actually, $Module_{1}$ has twice the execution time of $Module_{2}$ and $Module_{3}$. Setting their execution times to be exactly the same is difficult to achieve in our model, while it is also unnecessary. In practice we just need to make sure that the period of oscillator covers the time required for whole CRNs in each module to reach equilibrium , which can be guaranteed by controlling the parameters.\\ 
\\
Compared with our previous work, this paper extends the regulation mechanism applicable to two modules to any finite-multiple module regulation task. In principle, the regulation of multiple modules can be regarded as nesting of multiple sets of two-module regulation task, so that the corresponding clock signals can be constructed by our two-dimensional relaxation oscillation model. Since actual computational designs often involve the regulation of multiple modules (as we have demonstrated, even regulation of two modules is expanded to three modules to ensure accuracy), our generic design based on multiple modules is necessary. We sacrifice some precision in the estimation of the clock signal periods, although the results are still satisfactory from the simulation level. To enhance the accuracy, also in order to control the scale of whole oscillator model, we will try to build a better underlying oscillation structure in future work, including using a higher dimensional oscillator which can directly generate more groups of symmetrical clock signals to replace the two-dimensional relaxation oscillation, and find a more accurate calculation method of the periods.

\begin{ack}
This work was funded by the National Nature Science Foundation of China under Grant 12071428 and 62111530247, and the Zhejiang Provincial Natural Science Foundation of China under Grant LZ20A010002.
\end{ack}

\bibliography{ifacconf}             
                                                   







\end{document}